\documentclass[11pt,twoside,a4paper]{article}
\usepackage{amsmath,amssymb,amsthm}
\usepackage{hyperref}
\usepackage{graphicx,psfrag}

\newtheorem{thm}{Theorem}[section]
\newtheorem{lem}[thm]{Lemma}
\newtheorem{pro}[thm]{Proposition}
\newtheorem{cor}[thm]{Corollary}

\theoremstyle{definition}
\newtheorem{exa}[thm]{Example}

\theoremstyle{remark}
\newtheorem{rem}[thm]{Remark}

\newcommand{\R}{\mathbb{R}}

\newcommand{\N}{\mathbb{N}}

\newcommand{\cP}{\mathcal{P}}

\newcommand{\al}{\alpha}
\newcommand{\be}{\beta}
\newcommand{\ga}{\gamma}

\newcommand{\de}{\delta}
\newcommand{\De}{\Delta}
\newcommand{\ep}{\varepsilon}
\newcommand{\om}{\omega}

\newcommand{\si}{\sigma}

\renewcommand{\phi}{\varphi}

\newcommand{\CAT}{\operatorname{CAT}}

\newcommand{\id}{\operatorname{id}}

\newcommand{\intr}{\operatorname{int}}

\newcommand{\reg}{\operatorname{reg}}

\newcommand{\ay}{\operatorname{aY}}
\newcommand{\harm}{\operatorname{Harm}}
\newcommand{\hm}{\operatorname{Hm}}

\newcommand{\h}{\operatorname{h}}

\newcommand{\width}{\operatorname{width}}

\newcommand{\set}[2]{\{#1:\,\text{#2}\}}
\newcommand{\sm}{\setminus}
\newcommand{\sub}{\subset}

\newcommand{\ov}{\overline}

\newcommand{\wh}{\widehat}

\begin{document}

\title{SRA-free condition by Zolotov for self-contracted curves and nondegeneracy of
zz-distance for M\"obius structures on the circle}
\author{Sergei Buyalo\footnote{Supported by RFFI Grant
17-01-00128a}}

\date{}
\maketitle

 \begin{abstract} SRA-free condition for metric spaces (that is, spaces without Small Rough Angles)
was introduced by Zolotov to study rectifiability of self-contracted curves in
various metric spaces. We give a M\"obius invariant version of this notion which allows to show
that zz-distance associated with a respective M\"obius structure on the circle
is nondegenerate. This result is an important part of a solution to the inverse problem
of M\"obius geometry on the circle.
 \end{abstract}


\medskip


\section{Introduction}
\label{introduction}

A {\em M\"obius structure}
$M$
on a set 
$X$
is a class of M\"obius equivalent semi-metrics on
$X$,
where two semi-metrics are equivalent if and only if they have
the same cross-ratios on every 4-tuple of points in 
$X$.
A M\"obius structure usually lives on the boundary at infinity of a hyperbolic space.

The {\em inverse problem} of M\"obius geometry asks to describe M\"obius structures which are
induced by hyperbolic spaces. We are interested in the inverse problem for M\"obius
structures on the circle
$X=S^1$. 
A natural condition for a such structure
$M$ is formulated in 
\cite{Bu18} as {\em monotonicity axiom}

(M) given a 4-tuple
$q=(x,y,z,u)\in X^4$
such that the pairs
$(x,y)$, $(z,u)$
separate each other,
we have
$$|xy|\cdot|zu|>\max\{|xz|\cdot|yu|,|xu|\cdot|yz|\}$$
for some and hence any semi-metric from
$M$.

If we take one of the points of
$q=(x,y,z,u)$,
for example
$u$,
as infinitely remote for a semi-metric
$|\cdot,\cdot|_u\in M$,
then mononicity axiom tells us that
$$|xy|_u>\max\{|xz|_u,|yz|_u\},$$
where
$z$
is between
$x$, $y$
on
$X_u=X\sm\{u\}=\R$.
In particular, if a segment
$xz\sub X_u$
is contained in a segment
$xy\sub X_u$,
then
$|xz|_u<|xy|_u$.
Surprisingly, exactly this condition defines {\em self-contracted} curves in metric
spaces. Self-contracted curves usually appear as gradient curves of convex
functions, and they play an important role in a number of questions. Basic
problem for a self-contrated curve is to establish its rectifiability, see 
\cite{DDDL}, \cite{DDDR}, \cite{DLS}, \cite{LOZ}, \cite{Le}, \cite{Ohta},
\cite{ST}, \cite{Zo18}.

We consider the set 
$\hm$
of unordered harmonic pairs of unordered pairs 
of points in
$X$
as a required filling of a M\"obius structure
$M$
in the inverse problem of M\"obius geometry. Axiom~(M) allows to uniquely define a line 
$\h_a\sub\hm$
as the family of harmonic 4-tuples with a common axis, which is a pair of distinct points
$a=(x,y)$
on
$X$.
Every harmonic pair
$q=(a,b)\in\hm$
has two axes. Thus moving along of a line, we have a possibility
to change the axis of the line at any moment and move along the line
determined by the other axis. This leads to the notion of zig-zag path.
A {\em zig-zag} path, or zz-path, 
$S\sub\hm$
is defined as finite sequence of segments 
$\si_i$
on lines in
$\hm$,
where consecutive segments 
$\si_i$, $\si_{i+1}$
have a common end
$q=\si_i\cap\si_{i+1}\in\hm$
with axes determined by
$\si_i$, $\si_{i+1}$.

We define a distance
$\de$
on
$\hm$
as
$$\de(q,q')=\inf_S|S|,$$
where the infimum is taken over all zz-paths
$S\sub\hm$
from
$q$
to
$q'$,
and
$|S|$
is the length of
$S$
defined as the sum of the lengths of its sides. Its turns out that 
$\de$
is finite and satisfies all the triangle axioms except maybe the positivity one.
We replace axiom~(M) by a slightly stonger axiom~(M($\al$)). 

(M($\al$)) Fix  
$0<\al<1$.
Given a 4-tuple
$q=(x,y,z,u)\in X^4$
such that the pairs
$(x,y)$, $(z,u)$
separate each other, we have
$$|xy|\cdot|zu|\ge\max\{|xz|\cdot|yu|+\al|xu|\cdot|yz|,\al|xz|\cdot|yu|+|xu|\cdot|yz|\}$$
for some and hence any semi-metric from
$M$.

This is a
M\"obius invariant version of the SRA-free condition introduced by 
Zolotov in \cite{Zo18}.

Now, our main result is as follows.

\begin{thm}\label{thm:main} Let 
$M$
be a ptolemaic M\"obius structure of the circle
$X$
which satisfies axiom~(M($\al$)) for some
$0<\al<1$,
such that the 
$M$-topology is the topology of
$S^1$.
Then the distance
$\de$
on
$\hm$
is nondegenerate, 
$\de(q,q')\neq 0$
if and only if
$q\neq q'$,
the 
$\de$-metric 
topology 
on
$\hm$
coincides with one induced from
$X^4$,
and the metric space
$(\hm,\de)$
is complete. In particual,
$(\hm,\de)$
is a proper geodesic metric space.
\end{thm}

\section{M\"obius structures}
\label{sect:moebius_structures}

\subsection{Basic notions}
\label{subsect:basics}

Let
$X$
be a set. A 4-tuple
$q=(x,y,z,u)\in X^4$
is said to be {\em admissible} if no entry occurs three or
four times in
$q$.
A 4-tuple
$q$
is {\em nondegenerate}, if all its entries are pairwise
distinct. Let
$\cP_4=\cP_4(X)$
be the set of all ordered admissible 4-tuples of
$X$, $\reg\cP_4\sub\cP_4$
the set of nondegenerate 4-tuples.

A function
$d:X^2\to\wh\R=\R\cup\{\infty\}$
is said to be a {\em semi-metric}, if it is symmetric,
$d(x,y)=d(y,x)$
for each
$x$, $y\in X$,
positive outside the diagonal, vanishes on the diagonal
and there is at most one infinitely remote point
$\om\in X$
for
$d$,
i.e. such that
$d(x,\om)=\infty$
for some
$x\in X\sm\{\om\}$.
Moreover, we require that if
$\om\in X$
is such a point, then
$d(x,\om)=\infty$
for all 
$x\in X$, $x\neq\om$.
A metric is a semi-metric that satisfies the triangle inequality.

A {\em M\"obius structure}
$M$
on
$X$
is a class of M\"obius equivalent semi-metrics on
$X$,
where two semi-metrics are equivalent if and only if they have
the same cross-ratios on every
$q\in\reg\cP_4$.

Given
$\om\in X$,
there is a semi-metric 
$d_\om\in M$
with infinitely remote point
$\om$.
It can be obtained from any semi-metric
$d\in M$
for which 
$\om$
is not infinitely remote by a {\em metric inversion},
\begin{equation}\label{eq:metric_inversion}
d_\om(x,y)=\frac{d(x,y)}{d(x,\om)d(y,\om)}. 
\end{equation}
Such a semi-metric is unique up to a homothety, see \cite{FS13},
and we use notation
$|xy|_\om=d_\om(x,y)$
for the distance between
$x$, $y\in X$
in that semi-metric. We also use notation
$X_\om=X\sm\{\om\}$.

Every M\"obius structure
$M$
on
$X$
determines the 
$M$-{\em topology}
whose subbase is given by all open balls centered at finite points
of all semi-metrics from
$M$
having infinitely remote points.

\begin{exa}\label{exa:canonical_moebius_circle} Our basic example is the 
{\em canonical} M\"obius structure 
$M_0$
on the circle
$X=S^1$.
We think of
$S^1$
as the unit circle in the plane,
$S^1=\set{(x,y)\in\R^2}{$x^2+y^2=1$}$.
For 
$\om=(0,1)\in X$
the stereographic projection
$X_\om\to\R$
identifies
$X_\om$
with real numbers 
$\R$.
We let
$d_\om$
be the standard metric on
$\R$,
that is,
$d_\om(x,y)=|x-y|$
for any
$x,y\in\R$.
This generates a M\"obius structure on
$X$
which is called {\em canonical}. The basic feature of the canonical M\"obius 
structure on
$X=S^1$
is that for any 4-tuple
$(\si,x,y,z)\sub X$
with the cyclic order 
$\si xyz$
we have 
$d_\si(x,y)+d_\si(y,z)=d_\si(x,z)$.
\end{exa}

\subsection{Harmonic pairs}
\label{subsect:harm_pairs}

From now on, we assume that 
$X$
is the circle,
$X=S^1$.
It is convenient to use unordered pairs 
$(x,y)\sim(y,x)$
of distinct points on
$X$,
and we denote their set by
$\ay=S^1\times S^1\sm\De/\sim$,
where
$\De=\set{(x,x)}{$x\in S^1$}$
is the diagonal. A pair 
$q=(a,b)\in\ay\times\ay$
is harmonic if
\begin{equation}\label{eq:harmonic}
|xz|\cdot|yu|=|xu|\cdot|yz|
\end{equation}
for some and hence any semi-metric of the M\"obius structure, where
$a=(x,y)$, $b=(z,u)$.
The pair
$a$
is called the {\em left} axis of
$q$,
while
$b$
the {\em right} axis. We denote by
$\harm$
the set of harmonic pairs,
$\harm\sub\ay\times\ay$,
of the given M\"obius structure. There is a canonical involution 
$j:\harm\to\harm$
without fixed points given by
$j(a,b)=(b,a)$.
Note that
$j$
permutes left and right axes. The quotient space we denote by
$\hm:=\harm/j$.
In other words,  
$\hm$
is the set of unordered harmonic pairs of unordered pairs of points in
$X$.

\subsection{Axioms}
\label{subsect:axioms}

We list a set of axioms for a M\"obius structure
$M$ 
on the circle
$X=S^1$.

\begin{itemize}
 \item [(T)] Topology: $M$-topology
on
$X$
is that of
$S^1$.

\item[(M($\al$))] Monotonicity: Fix 
$0<\al<1$.
Given a 4-tuple
$q=(x,y,z,u)\in X^4$
such that the pairs
$(x,y)$, $(z,u)$
separate each other, we have

$$|xy|\cdot|zu|\ge\max\{|xz|\cdot|yu|+\al|xu|\cdot|yz|,\al|xz|\cdot|yu|+|xu|\cdot|yz|\}$$
for some and hence any semi-metric from
$M$.

\item[(P)] Ptolemy: for every 4-tuple
$q=(x,y,z,u)\in X^4$
we have
$$|xy|\cdot|zu|\le |xz|\cdot|yu|+|xu|\cdot|yz|$$
for some and hence any semi-metric from
$M$.
\end{itemize}

A M\"obius structure 
$M$
on the circle
$X$
that satisfies axioms T, M($\al$), P is said to be {\em strictly monotone}.

\begin{rem}\label{rem:zolotov} Axiom~M($\al$) is motivated by the work
\cite{Zo18} of V.~Zolotov.
\end{rem}

\begin{rem}\label{rem:axiom_Q} Axiom~P is satisfied, for example,
for the M\"obius structure on the boundary at infinity of any
$\CAT(-1)$
space, see \cite{FS12}.
\end{rem}

\begin{rem}\label{rem:canonical_axioms} The canonical M\"obius structure
$M_0$
on
$X=S^1$
clearly satisfies Axioms~T, M($\al$), P.
\end{rem}

We derive some immediate corollaries from the axioms. It follows from P 
that any semi-metric from
$M$
with an infinitely remote point is a metric, i.e. it satisfies the triange inequality.

\begin{cor}\label{cor:completeness} For every
$\om\in X$
the metric space
$X_\om$
is complete.
\end{cor}

\begin{proof} By Axiom~(T), 
$X_\om$
is homeomorphic to
$\R$.
Since
$\om$
is infinitely remote, the space
$X_\om$
is unbounded in the respective metric. These two properties together easily imply that
$X_\om$
is complete.
\end{proof}

A choice of
$\om\in X$
uniquely determines the interval
$xy\sub X_\om$
for any distinct
$x$, $y\in X$
different from
$\om$
as the arc in
$X$
with the end points
$x$, $y$
that does not contain
$\om$.
As an useful implication of Axiom~M($\al$) we have

\begin{cor}\label{cor:interval_monotone} Axiom M($\al$) implies the following.
Assume for a nondegenerate 4-tuple
$q=(x,y,z,u)\in\reg\cP_4$
the interval
$xz\sub X_u$
is contained in
$xy$, $xz\sub xy\sub X_u$.
Then
$|xz|_u<|xy|_u$.
That is, 
$X_u$
is self-contraced for any
$u\in X$.
\end{cor}

\begin{proof} By the assumption, the pairs
$(x,y)$, $(z,u)$
separate each other. Hence, by Axiom~M($\al$) we have
$|xz|_u|<|xy|_u$.
\end{proof}

\begin{cor}\label{cor:harm_separate} For any harmonic pair
$((x,y),(z,u))\in\harm$
the pairs
$(x,y)$, $(z,u)\in\ay$
separate each other.
\end{cor}

\begin{proof} In the metric from
$M$
with infinitely remote point
$u$
we have
$|xz|_u=|zy|_u$.
By Corollary~\ref{cor:interval_monotone},
$z$
lies between
$x$, $y$
on the line
$X_u$.
Hence, the pairs
$(x,y)$, $(z,u)$
separate each other.
\end{proof}

\section{Lines and zigzag paths}
\label{sect:lines_zzpath}

Here we briefly recall definitions and some properties of lines and zigzag paths
from \cite{Bu18}.

\subsection{Lines}
\label{subsect:lines}

\begin{lem}\label{lem:project_point_line} Given
$a\in\ay$
and
$x\in X$, $x\notin a$,
there is a uniquely determined
$y\in X$
such that the pair
$(a,b)$
is harmonic, 
$(a,b)\in\hm$,
where
$b=(x,y)$.
\end{lem}

\begin{proof} Let
$a=(z,u)$.
We take a metric from
$M$
with infinitely remote point
$u$.
The distance function
$y\mapsto|zy|_u$
is continuous on
$X_u$
(see \cite[Lemma~4.1]{Bu17}),
thus there is
$y\in X_u$
with
$|xz|_u=|zy|_u$.
Hence the pair
$(a,b)$,
where
$b=(x,y)$
is harmonic. By Corollary~\ref{cor:interval_monotone}, the point
$y$
is uniquely determined.
\end{proof}

We denote by
$\rho_a(x)=y$
the point
$y$
from Lemma~\ref{lem:project_point_line}. The {\em line} with axis
$a\in\ay$
is defined as the set  
$\h_a\sub\hm$
which consists of all pairs
$q=(a,b)$
with
$b=(x,\rho_a(x))$
where 
$x$
run over an arc in
$X$
determined by
$a$.
This is well defined because
$\rho_a:X\to X$
is involutive,
$\rho_a^2=\id$
(we extend
$\rho_a$
to
$a=(z,u)$
by
$\rho_a(z)=z$, $\rho_a(u)=u$). In this case, we use notation
$x_a:=b$
and say that
$x_a\in\h_a$
is the projection of
$x$
to the line
$\h_a$.

For more about lines see \cite{Bu18}. In particual, every line 
is homeomorphic to the real line 
$\R$,
different points on a line are in {\em strong causal relation}, that is,
either of them lies on an open arc in
$X$
determined by the other one, and vice versa, given 
$b$, $b'\in\ay$
in strong causal relation, there exists a unique line
$\h_a$
through
$b$, $b'$,
see \cite[Lemma~3.2, Lemma~4.2]{Bu18}. In this case, the pair
$a\in\ay$ 
(or the line
$\h_a$)
is called the {\em common perpendicular} to
$b$, $b'$.

The {\em segment}
$qq'$
of a line
$\h_a$
with
$q=(a,b)$, $q'=(a,b')\in\h_a$
is defined as the union of 
$q$, $q'$
and all 
$q''=(a,b'')\in\h_a$
such that 
$b''$
separates
$b$, $b'$.
The last means that 
$b$
and
$b'$
lie on different open arcs in
$X$
determined by
$b''$.
The points
$q$, $q'$
are the {\em ends} of
$qq'$.
The segment
$qq'\sub\h_a$
is homeomorphic to the standard segment
$[0,1]$.

\subsection{Distance between harmonic pairs with common axis}
\label{subsect:distance_harmonic_pairs}

Given two harmonic pairs 
$q$, $q'\in\hm$
with a common axis, say
$q=(a,b)$
and
$q'=(a,b')$,
we define {\em the distance}
$|qq'|$
between them as
\begin{equation}\label{eq:distance}
|qq'|=\left|\ln\frac{|xz'|\cdot|yz|}{|xz|\cdot|yz'|}\right|
\end{equation}
for some and hence any semi-metric on
$X$
from
$M$,
where
$a=(x,y)$, $b=(z,u)$, $b'=(z',u')\in\ay$.

One easily checks that every line
$\h_a\sub\hm$
with this distance is isometric to the real line
$\R$
with the standard distance.

\subsection{Width of a strip}
\label{subsect:width_strip}

We say that a 4-tuple
$p=(a,b)\in X^4$
with
$a=(x,y)$, $b=(u,z)$
is a {\em strip} if
$a$, $b$
are in the strong causal relation and the pairs
$(x,z)$, $(u,y)$
separate each other. Note that
$p'=(b,c)\in X^4$
with
$b=(x,u)$, $c=(y,z)$
is also a strip based on the same 4-tuple
$(x,y,u,z)\in X^4$.

Since the pairs
$a$, $b$
are in the strong causal relation, there is uniquely determined
common perpendicular
$s=(v,w)$
to
$a$, $b$.
We assume that 
$w$
lies on the arc in 
$X$
determined by
$b$
which does not contain
$a$.
We denote
$g=|vx|_w$, $h=|vu|_w$.
Then
$|vy|_w=g$, $|vz|_w=h$
because
$s$
is the common perpendicular to
$a$, $b$,
and
$h>g$
by mononicity and the choice of the infinitely remote point
$w$.
We use notation
$p=(a,b,s)$
for a strip with common perpendicular
$s$.
Note that 
$s$
is uniquely determined by
$(a,b)$,
and we add 
$s$
to fix notation. A strip
$p=(a,b,s)$
is said to be {\em narrow} if
$|xu|_w$, $|yz|_w\le g=|vx|_w=|vy|_w$.

We define the width of the strip
$p$
as the length
$l=\width(p)$
of the segment
$x_su_s=y_sz_s\sub\h_s$
on the line
$\h_s$.

\begin{lem}\label{lem:width_strip_above} For any strip
$p=(a,b,s)$
we have
$$\width(p)\le 2\sqrt{\frac{|xu||yz|}{|xy||zu|}},$$
where
$a=(x,y)$, $b=(u,z)$.
\end{lem}

\begin{lem}\label{lem:width_strip_below} Let
$p=(a,b,s)$
be a strip with
$a=(x,y)$, $b=(u,z)$, $s=(v,w)$.
Then for the width
$l=\width(p)$
of
$p$
we have
$$2\sinh(l/2)\ge\al(1+\al)\sqrt{\frac{|xu||yz|}{|xy||uz|}}$$
and
$$|xu|_w,|yz|_w\le\frac{e^l-1}{\al(1+\al)}|xy|_w,\quad 
\al\le\frac{|xu|_w}{|yz|_w}\le\frac{1}{\al},$$
where
$\al$
is the constant from Axiom M($\al$).
\end{lem}

\begin{proof}[Proof of Lemmas~\ref{lem:width_strip_above} and \ref{lem:width_strip_below}]
In the metric
$|\ |_w$
on
$X_w$
with infinitely remote point
$w$
we have 
$|xu|_w$, $|yz|_w\ge h-g$
and also
$|xy|_w\le 2g$, $|uz|_w\le 2h$
by the triange inequality. Thus
$$\frac{|xu|_w|yz|_w}{|xy|_w|uz|_w}\ge\frac{(h-g)^2}{4gh}=\sinh^2(l/2)\ge (l/2)^2,$$
where
$l=\width(p)$,
because
$e^l=h/g$,
see Eq.~(\ref{eq:distance}). This proves Lemma~\ref{lem:width_strip_above}.

The pairs
$(v,z)$
and
$(w,y)$
separate each other. Thus by Axiom M($\al$) we have
$$|vz||wy|\ge\max\{|yz||vw|+\al|zw||vy|,\al|yz||vw|+|zw||vy|\}$$
in any semi-metric of the M\"obius structure M. Taking the metric with
infinitely remote point
$w$,
we obtain
\begin{equation}\label{eq:axiom_estimate}
h\ge\max\{|yz|_w+\al g,\al|yz|_w+g\}.
\end{equation}
The inequality (\ref{eq:axiom_estimate})
implies
$\al|yz|_w\le h-g$.
Similarly
$\al|ux|_w\le h-g$.

Again, by Axiom~M($\al$),
$|xy|_w\ge|xv|_w+\al|vy|_w=(1+\al)g$
and similarly
$|uz|_w\ge(1+\al)h$.
Thus
$$\frac{|yz||xu|}{|xy||uz|}\le\frac{(h-g)^2}{\al^2(1+\al)^2gh}.$$
Since
$e^l=h/g$,
this gives
$$4(\sinh l/2)^2\ge\al^2(1+\al)^2\frac{|yz||xu|}{|xy||uz|},$$
hence the first estimate of Lemma~\ref{lem:width_strip_below}.

Estimates above also give
$|ux|_w\le\frac{1}{\al}(h-g)=\frac{e^l-1}{\al}g=\frac{e^l-1}{\al(1+\al)}|xy|_w$
and similarly
$|yz|_w\le\frac{e^l-1}{\al(1+\al)}|xy|_w$.
Using the inequalities
$|yz|_w$, $|xu|_w\ge h-g$
together with the previous estimates from above, we obtain
the second estimate of Lemma~\ref{lem:width_strip_below}.
\end{proof}

\subsection{Zigzag paths}
\label{subsect:zigzag_paths}

Every harmonic pair
$q=(a,b)\in\hm$
has two axes. Thus moving along of a line, we have a possibility
to change the axis of the line at any moment and move along the line
determined by the other axis. This leads to the notion of zig-zag path.
A {\em zig-zag} path, or zz-path, 
$S\sub\hm$
is defined as finite (maybe empty) sequence of segments 
$\si_i$
in
$\hm$,
where consecutive segments 
$\si_i$, $\si_{i+1}$
have a common end
$q=\si_i\cap\si_{i+1}\in\hm$
with axes determined by
$\si_i$, $\si_{i+1}$.
Segments 
$\si_i$
are also called {\em sides} of
$S$,
while a {\em vertex} of
$S$
is an end of a side. Given 
$q$, $q'\in\hm$,
there is a zz-path
$S$
in
$\hm$
with at most five sides that connects
$q$
and
$q'$
(see \cite[Lemma~3.3]{Bu18}).

\section{Metric on $\hm$}
\label{sect:metric}

\subsection{Distance $\de$ on $\hm$}
\label{subsect:dist_de}

Let
$S=\{\si_i\}$
be a zz-path in
$\hm$.
We define the length of
$S$
as the sum
$|S|=\sum_i|\si_i|$
of the length of its sides. Now, we define a distance
$\de$
on
$\hm$
by
$$\de(q,q')=\inf_S|S|,$$
where the infimum is taken over all zz-paths
$S\sub\hm$
from
$q$
to
$q'$.

One easily sees that 
$\de$
is a finite pseudometric on
$\hm$,
see \cite[Proposition~6.2]{Bu18}.
Now, we show that
$\de$
is a metric.

\begin{lem}\label{lem:ratio_distortion} Fix a harmonic
$q\in\hm$, $q=(a,s)$
with
$a=(x,y)$, $s=(v,w)\in\ay$.
Let
$p'=(a',b',s')$
be a strip with
$a'=(x',y')$, $b'=(u',z')$, $s'=(v',w')$
such that
$w$
lies on the arc in
$X$
determined by
$b'$
that contains
$w'$.
Then
$$\al\be\le\frac{|x'u'|_w}{|y'z'|_w}\le\frac{1}{\al\be}$$
for some
$\be=\be(p',w)<1$.
If, in addition 
$a'$, $b'\in U_\ep(a)$, $s'\in V_\ep(s)$
for a sufficiently small
$\ep$, $0<\ep<1/8$,
where
$U_\ep(a)=\set{(x',y')\in\ay}{$|xx'|_w<\ep,\ |yy'|_w<\ep$}$,
$V_\ep(s)=\set{(v',w')\in\ay}{$|vv'|_x<\ep,\ |ww'|_x<\ep$}$,
then
$\be\ge 1-8\ep$.
In this case,
$$2\sinh(l'/2)\ge c\max\{|x'u'|_w,|y'z'|_w\}$$
with
$c=\frac{\al(1+\al)\sqrt{\al}}{2}\sqrt{1-8\ep}$,
where
$l'=\width(p')$.
\end{lem}

\begin{proof} Using the metric inversion (\ref{eq:metric_inversion}), we have
$$|x'u'|_w=\frac{|x'u'|_{w'}}{|x'w|_{w'}|u'w|_{w'}}\quad\textrm{and}\quad
  |y'z'|_w=\frac{|y'z'|_{w'}}{|y'w|_{w'}|z'w|_{w'}}.$$
Thus
$$\frac{|x'u'|_w}{|y'z'|_w}=\ga\frac{|x'u'|_{w'}}{|y'z'|_{w'}},$$
where
$$\ga=\frac{|y'w|_{w'}|z'w|_{w'}}{|x'w|_{w'}|u'w|_{w'}}.$$
By Lemma~\ref{lem:width_strip_below}
\begin{equation}\label{eq:distor_estimate}
2\sinh(l'/2)\ge\al(1+\al)\sqrt{\frac{|x'u'||y'z'|}{|x'y'||u'z'|}}\quad\textrm{and}\quad
 \al\le\frac{|x'u'|_{w'}}{|y'z'|_{w'}}\le\frac{1}{\al}.
\end{equation}
Using that
$|xw|_{w'}=1/|xw'|_w$,
we obtain
$$\ga=\frac{|x'w'|_w|u'w'|_w}{|y'w'|_w|z'w'|_w}.$$
If 
$w'=w$
then
$\ga =1$.
Otherwise, consider two cases:

(1) $w$
lies on the arc in
$X$
between
$w'$
and
$z'$
that does not contain
$y'$.
In this case
$u'w'\sub x'w'$, $y'w'\sub z'w'$, $x'w'\sub y'w'$
in
$X_w$.
Thus
$|u'w'|_w<|x'w'|_w$, $|y'w'|_w<|z'w'|_w$, $|x'w'|_w<|y'w'|_w$
and hence
$$\ga\le\frac{|x'w'|_w^2}{|y'w'|_w^2}<1.$$
In this case, we put
$\be=\ga$
and obtain
$$\al\be\le\frac{|x'u'|_w}{|y'z'|_w}\le\frac{1}{\al\be}$$
because
$\ga<1/\ga$.

(2) $w$
lies on the arc in
$X$
between
$w'$
and
$u'$
that does not contain
$x'$.
In this case
$x'w'\sub u'w'$, $z'w'\sub y'w'$, $z'w'\sub x'w'$
in
$X_w$.
Thus
$|u'w'|_w>|x'w'|_w$, $|z'w'|_w<|y'w'|_w$, $|x'w'|_w>|y'w'|_w$
and hence
$$\ga\ge\frac{|x'w'|_w^2}{|y'w'|_w^2}>1.$$
In this case, we put
$\be=1/\ga$
and obtain
$$\al\be\le\frac{|x'u'|_w}{|y'z'|_w}\le\frac{1}{\al\be}$$
because
$\ga>1/\ga$.

Finally, the condition
$s'\in V_\ep(s)$
implies that
$|w'w|_x<\ep$
and hence
$|w'x|_w>1/\ep$
by the metric inversion (\ref{eq:metric_inversion}).
We put
$t:=|w'x|_w$.

The condition
$a'$, $b'\in U_\ep(a)$
implies
$$t-\ep\le |w'x'|_w, |w'u'|_w\le t+\ep$$
and
$$|w'y|_w-\ep\le|w'y'|_w,|w'z'|_w\le |w'y|_w+\ep,$$
by the triange inequality.

We can assume without loss of generality that
$|xy|_w=1$.
Since 
$w'$
lies on the arc in
$X$
determined by
$a$
that contains
$w$,
we have
$t-1\le|w'y|_w\le t+1$
and hence
$t-1-\ep\le|w'y'|_w\le t+1+\ep$.
Using that
$t>1/\ep$
we obtain
$$\ga\ge\frac{(t-\ep)^2}{(t+1+\ep)^2}\ge\frac{(1-\ep^2)^2}{(1+\ep+\ep^2)^2}\ge (1-2\ep)^2\ge 1-8\ep$$
and similarly
$$\ga\le\frac{(t+\ep)^2}{(t-1-\ep)^2}\le\frac{(1+\ep^2)^2}{(1-\ep-\ep^2)^2}\le (1+2\ep)^2\le 1+8\ep.$$
For the first case, 
$\be=\ga$,
we take the first estimate above and obtain
$\be\ge 1-8\ep$.
For the second case,
$\be=1/\ga$,
we take the second estimate above and obtain
$\be\ge 1/(1+8\ep)\ge 1-8\ep$.

The condition
$a'$, $b'\in U_\ep(a)$
also implies
$|x'y'|_w$, $|u'z'|_w\le 1+2\ep$.
Using (\ref{eq:distor_estimate}), we obtain
$$2\sinh(l'/2)\ge c\max\{|x'u'|_w,|y'z'|_w\}$$
with
$c=\frac{\al(1+\al)\sqrt{\al}}{2}\sqrt{1-8\ep}$.
\end{proof}

\begin{lem}\label{lem:zz_path_length_below} Given
$\ep$, $0<\ep\le1/16$,
there is
$t=t(\ep)>0$
such that for every zz-path
$S$
in
$\hm$
with 
$|S|<t$
and the initial vertex
$q=q_1=(a,s)$, $a=(x,y)$, $s=(v,w)\in\ay$,
the vertices 
$q_i=(a_i,s_i)$, $i\ge 1$,
of
$S$
lie in
$U_\ep(a)\times V_\ep(s)$,
where
$U_\ep(a)=\set{(x',y')\in\ay}{$|xx'|_w<\ep, |yy'|_w<\ep$}$, 
$V_\ep(s)=\set{(v',w')\in\ay}{$|vv'|_x<\ep, |ww'|_x<\ep$}$,
and
$$|q_iq_{i+1}|\ge c\max\{|x_ix_{i+1}|_w,|y_iy_{i+1}|_w\}$$
for every odd
$i\ge 1$,
where
$a_i=(x_i,y_i)\in\ay$,
$$|q_iq_{i+1}|\ge c\max\{|v_iv_{i+1}|_x,|w_iw_{i+1}|_x\}$$
for every even
$i\ge 2$,
where
$s_i=(v_i,w_i)\in\ay$,
and some constant
$c=c(\al)>0$.
\end{lem}

\begin{proof} We put
$c=\frac{\al(1+\al)\sqrt{\al}}{16}$
and take
$t>0$
such that 
$\sinh(t/2)\le c\ep$.

Denote
$s_i=(v_i,w_i)\in\ay$, $i\ge 1$, $s_1=s$.
Assume without loss of generality that the first side
$\si_1=q_1q_2$
of
$S$
lies on the line
$\h_{s_1}$.
Then we have
$s_{i+1}=s_i$
for odd 
$i$,
and
$a_{i+1}=a_i$
for even
$i$.

We argue by induction on the number 
$n$
of sides of
$S$.
The base of the induction:
$n=1$.

We have
$q_1=q=(a,s)$
with
$a=(x,y)$, $s=(v,w)\in\ay$, $q_2=(a_2,s)$
with
$a_2=(x_2,y_2)$,
and 
$s$
is the common perpendicular to the strip
$p_1=(a,a_2)$.
The width of
$p_1$
equals the length
$l_1$
of the side
$q_1q_2$, $|q_1q_2|=l_1=\width(p_1)$.
We assume that
$w$
lies on the arc in
$X$
determined by
$a_2$
that does not contain
$a$,
and that
$p_1$
is narrow, that is,
$|xx_2|_w$, $|yy_2|_w\le |vx|_w=|vy|_w$,
see sect.~\ref{subsect:width_strip}. Note that 
$|vx_2|_w\ge|vx|_w+\al|xx_2|_w$
by Axion~M($\al$). Thus
$e^{l_1}=|vx_2|_w/|vx|_w\ge 1+\al(|xx_2|_w/|vx|_w)$.
Hence if
$l_1\le\ln(1+\al)$,
then
$|xx_2|_w\le|vx|_w$
and similarly
$|yy_2|_w\le|vy|_w$,
i.e.,
$p_1$
is narrow.

By Lemma~\ref{lem:width_strip_below} we have
$$2\sinh(l_1/2)\ge\al(1+\al)\sqrt{\frac{|xx_2||yy_2|}{|xy||x_2y_2|}}$$
and
$$\al\le\frac{|xx_2|_w}{|yy_2|_w}\le\frac{1}{\al}.$$
Taking the normalization
$|xy|_w=1$
for the metric of
$X_w$,
we obtain
$|x_2y_2|_w\le 1+|xx_2|_w+|yy_2|_w\le 1+2|vx|_w\le 2$
by the triange inequality. Hence
$2\sinh(l_1/2)\ge\frac{\al(1+\al)\sqrt{\al}}{2}|xx_2|_w$.
We see that
$$\max\{|xx_2|_w,|yy_2|_w\}\le\frac{4}{\al(1+\al)\sqrt{\al}}\sinh(l_1/2)\le\ep/4$$
for 
$|S|=l_1<t$.
Thus
$q_i$
lies in
$U_{\ep/2}(a)\times\{s\}\sub U_{\ep/2}(a)\times V_{\ep/2}(s)$
for 
$i=1,2$,
and
$$l_1=|q_1q_2|\ge c\max\{|x_1x_2|_w,|y_1y_2|_w\},$$
because
$t\le 1/64$
by our assumptions
$c\le 1/8$, $\ep\le 1/16$,
and thus
$l\ge\sinh(l/2)$
for
$l\le t$.

Assume the lemma is proved for all
$n<k$,
and consider a zz-path 
$S$
with
$k$
sides and
$|S|<t$.
By the inductive assumption,
$q_i\in U_{\ep/2}(a)\times V_{\ep/2}(s)$
for 
$i=1,\dots,k$.
We have to show that the inclusion above holds also for 
$i=k+1$,
and assuming that
$k$
is odd, that
$$|q_kq_{k+1}|\ge c\max\{|x_kx_{k+1}|_w,|y_ky_{k+1}|_w\}.$$ 
Let
$p_k=(q_k,q_{k+1})$
be the 
$k$th
strip with the common perpendicular
$s_k=(v_k,w_k)$.
The condition
$q_k\in U_{\ep/2}(a)\times U_{\ep/2}(s)$
means that 
$a_k\in U_{\ep/2}(a)$
and
$s_k\in V_{\ep/2}(s)$.
Consider a parameterization
$q_k(\tau)$, $0\le\tau\le 1$,
of the segment
$q_kq_{k+1}\sub\h_{s_k}$
with
$q_k(0)=q_k$, $q_k(1)=q_{k+1}$.
Then
$q_k(\tau)=(a_k(\tau),s_k)$
and the map 
$\tau\mapsto a_k(\tau)\in\ay$
is continuous. The assumption
$q_{k+1}\not\in U_{\ep/2}(a)\times U_{\ep/2}(s)$
would imply that there is
$\tau'\in (0,1]$
with
$\max\{|x_kx_k(\tau')|_w,|y_ky_k(\tau')|_w\}=\ep/2$,
where
$a_k(\tau)=(x_k(\tau),y_k(\tau))$.
Since still
$q_k(\tau')\in U_\ep(a)\times V_\ep(s)$,
we can apply Lemma~\ref{lem:ratio_distortion} to conclude that
$$\al\be\le\frac{|x_kx_k(\tau')|_w}{|y_ky_k(\tau')|_w}\le\frac{1}{\al\be}$$
with
$1>\be\ge\sqrt{1-8\ep}\ge 1/2$,
and
$$\sinh(|q_kq_k(\tau')|/2)\ge \frac{\al(1+\al)\sqrt{\al}}{8}\max\{|x_kx_k(\tau')|_w,|y_ky_k(\tau')|_w\}.$$
Using that
$|q_kq_k(\tau')|\le|S|<t$,
we see that
$$\max\{|x_kx_k(\tau')|_w,|y_ky_k(\tau')|_w\}\le\frac{1}{2c}\sinh(|S|/2)<\ep/2$$
in contradiction with the assumption. Thus
$q_{k+1}\in U_{\ep/2}(a)\times V_{\ep/2}(s)$.
Applying Lemma~\ref{lem:ratio_distortion}, we obtain
$$|q_kq_{k+1}|\ge c\max\{|x_kx_{k+1}|_w,|y_ky_{k+1}|_w\}.$$
Arguments for even
$k$
are similar with replacing
$a\leftrightarrow s$
in Lemma~\ref{lem:ratio_distortion}.
\end{proof}

\begin{pro}\label{pro:nondegenerate_delta_distance} The distance
$\de$
on
$\hm$
is nondegenerate, if
$\de(q,q')=0$,
then
$q=q'$
for any
$q$, $q'\in\hm$.
\end{pro}

\begin{proof} Assume that 
$\de(q,q')=0$
for some
$q$, $q'\in\hm$.
It means that for every
$k\in\N$
there is a zz-path
$S=S(k)$
between
$q$, $q'$
with
$|S|<1/k$.
 
Let
$q=(a,s)$,
where
$a=(x,y)$, $s=(v,w)$.
We fix a sufficiently small
$\ep$, $0<\ep<1/16$,
and take
$t=t(\ep)>0$
as in Lemma~\ref{lem:zz_path_length_below}.
Then for every
$k\in\N$, $k\ge 1/t$,
the zz-path
$S(k)$
satisfies the conclusion of Lemma~\ref{lem:zz_path_length_below}.
In particular, all vertices
$q_i(k)$
of
$S(k)$
lie in
$U_\ep(a)\times V_\ep(s)$.
Thus
$q'=q$
because
$\ep$
is taken arbitrarily.
\end{proof}

\subsection{Completeness of the distance $\de$}
\label{subsect:completeness}

\begin{lem}\label{lem:prescribed_parabolic_shift} Given a harmonic 4-tuple
$q=(a,s)$
with
$a=(x,y)$, $s=(v,w)$,
and 
$v'$
lying on the open arc 
$d$
in
$X$
determined by
$a$
that does not contain
$w$,
there is a harmonic
$q'=(a',s')$
with
$s'=(v',w)$,
which is connected with
$q$
by a zz-path 
$S$
with at most 3 sides. Moreover, if
$s'\in V_\ep(s)$
for a sufficiently small
$\ep>0$,
then
$|S|<\ga$
with
$\ga=\ga(\ep)\to 0$
as
$\ep\to 0$,
where
$V_\ep(s)=\set{(v',w')\in\ay}{$|vv'|_x<\ep, |ww'|_x<\ep$}$.
\end{lem}

\begin{proof} We construct a zz-path
$S=\{\si_i\}$, $i=1,2,3$,
as follows. We define
$\si_1=ss_1\sub\h_a$
with
$s_1=(v_1,w_1)$
so that 
$v'_a\in\si_1$
(for the notation
$v'_a$
see sect.~\ref{subsect:lines}). Moreover, we can assume that 
$v'_a\in\intr\si_1$
unless
$v'=v$,
in which case
$s_1:=s$,
and the construction of
$S$
is finished with
$\si_1=\si_2=\si_3=q$
degenerated. Then
$\si_1$
is the segment on the line 
$\h_a$
between
$q$ 
and 
$q_1\in\hm$, $q_1=(a,s_1)$.

By construction, the pairs
$s_1$, $s'$
are in the strong causal relation and thus there is a (unique!) common
perpendicular
$a_1=(x_1,y_1)$
to them. We assume that
$v_1$
lies on the arc 
$d$
in
$X$,
and
$x_1$
lies on the arc in
$X$
determined by
$s$
that contains
$x$.

Next, we define
$\si_2:=aa_1\sub\h_{s_1}$
as the segment between
$q_1$
and
$q_2=(a_1,s_1)$,
and finally, we define
$\si_3\sub\h_{a_1}$
as the segment
$\si_3=(v_1)_{a_1}v'_{a_1}=(w_1)_{a_1}w_{a_1}$
between
$q_2$
and
$q_3=(a_1,s')=q'$.
In particular,
$a'=a_1$.
Thus the segment
$\si_1\sub\h_a$
uniquely determines our zz-path
$S=\si_1\si_2\si_3$
connecting
$q$
and
$q'$.

To prove the last assertion of the lemma we parametrize the line
$\h_a\sub\hm$, $h_a=\h_a(\tau)$, $\tau\in\R$,
so that
$\h_a(0)=q$, $\h_a(+\infty)=x$, $\h_a(-\infty)=y$,
and
$|\h_a(\tau)\h_a(\tau')|=|\tau-\tau'|$.
The problem with the zz-path
$S$
is that there is no constructive way to find the common perpendicular
to given lines. We avoid this problem by defining a map 
$f:R_+^2\to d$
as follows. For 
$\tau\in\R$
we take the point
$\h_a(\tau)=:(a,s_1)$
on the line
$\h_a$.
Next, we move along the line
$\h_{s_1}$
by the distance
$t\ge 0$
in the direction 
$v_1\in d$, 
where
$s_1=(v_1,w_1)$.
This gives us the point
$(a_1,s_1)$
on the line
$\h_{s_1}$.
Finally, we move along the line
$\h_{a_1}$
until we reach
$(v',w)\in\h_{a_1}$
for some 
$v'\in d$.
Then we define
$f(\tau,t):=v'$.
By the construction,
$a_1\sub\ov d$
(the closure of the arc 
$d$
in
$X$),
thus the projection
$w_{a_1}\in\h_{a_1}$
is well defined, hence the map 
$f$
is well defined.

Note that
$f(0,t)=v=f(\tau,0)$
for each
$\tau\in\R$, $t\ge 0$,
and that
$f(\tau,t)\neq v$
for any
$\tau\neq 0$, $t>0$,
because otherwise we would have two different common perpendiculars to
the distinct lines
$\h_s$, $\h_{s_1}$.
The map 
$f$
is continuous which follows from uniqueness in Lemma~\ref{lem:project_point_line}.
For 
$\ga>0$
we define
$v_\ga^\pm=f(\pm\ga,\ga)$, $\ep_\ga=\min\{|vv_\ga^\pm|_x\}$.
Since the set 
$$A_\ga=\set{(\tau,t)}{$-\ga\le\tau\le\ga,\ 0\le t\le\ga $}\sub\R_+^2$$
is connected, for every
$0\le\ep\le\ep_\ga$, $(v',w)\in V_\ep(s)$,
there is
$(\tau,t)\in A_\ga$
such that 
$f(\tau,t)=v'$.
Now for any
$0\le\ep\le\ep_1$
and
$(v',w)\in V_\ep(s)$
we take the minimal
$\ga=\ga(\ep)$
such that
$f(\tau,t)=v'$
for 
$(\tau,t)\in A_\ga$.
We have
$|\si_1|$, $|\si_2|\le\ga$,
and clearly,
$\ga(\ep)\to 0$
as
$\ep\to 0$.

Using Lemma~\ref{lem:width_strip_above}, we find
$|\si_3|\le 2\sqrt{\frac{|v_1v'||ww_1|}{|v_1w_1||v'w|}}$,
where
$v'=f(\tau,t)$.
Assuming that
$\tau>0$,
we normalize the metric of
$X_x$
by the condition
$|vw|_x=1$
(in the case
$\tau<0$
we take
$X_y$
instead of
$X_x$).
Then
$|v_1w_1|_x, |v'w|_x\ge|vw|_x=1$.
By Lemma~\ref{lem:width_strip_below}
$$|vv_1|_x,|ww_1|_x\le\frac{e^\ga-1}{\al(1+\al)}|vw|_x=\frac{e^\ga-1}{\al(1+\al)}.$$
By mononicity,
$|v'v_1|_x\le|vv_1|_x$.
Thus
$|\si_3|\le\frac{2(e^\ga-1)}{\al(1+\al)}\le c\ga$
for some constant
$c=c(\al)>0$
as
$\ep\to 0$.
Hence
$S=|\si_1|+|\si_2|+|\si_3|\le (2+c)\ga$,
which completes the proof.
\end{proof}

\begin{pro}\label{pro:de_metric_topology} The 
$\de$-metric
topology on
$\hm$
coincides with that induced from
$X^4$.
\end{pro}

\begin{proof} Let
$\tau$
be the topology on
$\hm$
induced from
$X^4$, $\tau_\de$
the 
$\de$-metric
topology. It follows from Lemma~\ref{lem:zz_path_length_below} that
every 
$\tau$-open
neighborhood
$U_\ep(a)\times V_\ep(s)$
of
$q=(a,s)\in\hm$
contains a 
$\de$-metric 
ball
$B_t(q)$
centered at
$q$,
where
$t=t(\ep)$.
Hence,
$\tau$
is not finer than
$\tau_\de$,
i.e.,
every
$\tau$-open
set in
$\hm$
is
$\tau_\de$-open.

Conversely, given a
$\de$-metric
ball
$B_t(q)$
of radius
$t>0$
centered at
$q=(a,s)\in\hm$
we find a sufficiently small
$\ep>0$
such that every point
$q'\in U_\ep(a)\times V_\ep(s)$, $q'=(a',s')$,
is connected with
$q$
by a zz-path of length
$<t$. 
We find a short zz-path
$S$
between
$q$
and
$q'$
as follows. First, we take
$q_1=(a_1,s_1)\in\hm$
with
$s_1=(v',w)\in V_\ep(s)$,
where
$s=(v,w)$, $s'=(v',w')$,
and connect 
$q$
with
$q_1$
by a zz-path
$S_1$
as in Lemma~\ref{lem:prescribed_parabolic_shift} with
$|S_1|<\ga_1(\ep)$.
Applying the same construction, we find a zz-path
$S_2$
between
$q_1$
and
$q_2=(a_2,s')$
with
$|S_2|<\ga_2(\ep)$.
Then
$s'$
is the common perpendicular to
$a'$, $a_2$,
and we connect
$q_2$
with
$q'$
along
$\h_{s'}$.

Recalling the construction of Lemma~\ref{lem:prescribed_parabolic_shift},
we see that the width of the strip between
$a$
and
$a_1$
is at most
$\ga_1(\ep)$,
and the width of the strip between
$a_1$
and
$a_2$
is at most
$\ga_2(\ep)$.
Combining with Lemma~\ref{lem:width_strip_below}, this allows to conclude
that
$a_2\in U_{\ga_3(\ep)}(a)$
for some
$\ga_3(\ep)$
with
$\ga_3(\ep)\to 0$
as
$\ep\to 0$.
By Lemma~\ref{lem:width_strip_above}, we find that
$|q_2q'|\le\ga_4(\ep)$
with
$\ga_4(\ep)\to 0$
as
$\ep\to 0$.
It follows that 
$\de(q,q')\le\ga(\ep)$
for some
$\ga(\ep)$
with
$\ga(\ep)\to 0$
as
$\ep\to 0$.
Taking 
$\ep>0$
sufficiently small, we find that any
$q'\in U_\ep(a)\times V_\ep(s)$
lies in
$B_t(q)$.
Hence
$\tau=\tau_\de$.
\end{proof}

\begin{pro}\label{pro:complete_de_distance} The distance
$\de$
on
$\hm$
is complete.
\end{pro}

\begin{proof} Let
$\{q_i\}\in\hm$
be a Cauchy sequence with respect to the metric
$\de$. 
For a given
$t>0$
we have
$\de(q_i,q_j)<t$
for all sufficiently large 
$i$, $j$.
By approximation, we find for each
$i$, $j$
a zz-path
$S_{ij}$
between
$q_i$, $q_j$
with
$|S_{ij}|\le 2\de(q_i,q_j)$.
Then for all sufficiently large
$i$, $j$
we have
$\de(q_i,q_j)<t/2$
and thus
$|S_{ij}|<t$.

Now, fix 
$\ep>0$, $\ep\le 1/16$,
and apply Lemma~\ref{lem:zz_path_length_below} to the zz-paths
$S_{ij}$
for a fixed
$i$ 
and all sufficiently large
$j$.
Then
$q_j\in U_\ep(a_i)\times V_\ep(s_i)$,
where
$q_i=(a_i,s_i)$.
Moreover, the estimates of Lemma~\ref{lem:zz_path_length_below}
(and the triange inequality) imply that the sequences
$x_j$, $y_j\in X_w$, $v_j$, $w_j\in X_x$
are Cauchy with respect to the metrics
$X_w$, $X_x$,
where
$a_j=(x_j,y_j)$, $s_j=(v_j,w_j)$.
By Corollary~\ref{cor:completeness}, there are limits
$x=\lim x_j$, $y=\lim y_j$, $v=\lim v_j$, $w=\lim w_j$,
and the 4-tuple
$q=(x,y,v,w)$
is harmonic. By Proposition~\ref{pro:de_metric_topology},  
$\de(q,q_j)\to 0$.
\end{proof}

\begin{proof}[Proof of Theorem~\ref{thm:main}] It follows from 
Propositions~\ref{pro:nondegenerate_delta_distance},
\ref{pro:de_metric_topology} and \ref{pro:complete_de_distance}
that the distance
$\de$
on
$\hm$
is nondegenerate, that the 
$\de$-metric
topology coincides with that induced from
$X^4$,
in particual, the metric space 
$(\hm,\de)$
is locally compact, and that 
$(\hm,\de)$
is complete. By definition,
$(\hm,\de)$
is a length space. By Hopf-Rinov theorem,
$(\hm,\de)$
is proper, that is, closed balls are compact. A standard argument shows that
$(\hm,\de)$
is geodesic, i.e., for any two points 
$p$, $q$
with
$\de(p,q)<\infty$
there is a geodesic
$pq$
between them.
\end{proof}

\end{document}